\documentclass[12pt]{amsart}
\usepackage[english]{babel}
\usepackage{amsmath,amssymb,amsthm,latexsym,amsfonts}
\usepackage[mathscr]{eucal}
\usepackage{graphicx}
\usepackage{tikz}
\usepackage{longtable,array}
\theoremstyle{plain}
\newtheorem{lemma}{Lemma}
\newtheorem{theorem}[lemma]{Theorem}

\newtheorem{definition}{Definition}
\newtheorem{remark}{Remark}

\newcommand{\C}{\mathcal C}

\newcommand{\G}{\Gamma}

\newcommand{\matr} [4] {\left(\begin{array}{@{}c@{\ }c@{}} #1 & #2 \\ #3 & #4 \\ \end{array} \right)}

\newcommand{\seifuno}[3]{\big(#1,({\scriptstyle #2},{\scriptstyle #3})\big)}

\newcommand{\seifdue}[5]{\big(#1,({\scriptstyle #2},{\scriptstyle #3}),
                       ({\scriptstyle #4},{\scriptstyle #5})\big)}

\newcommand{\bigu}[4]{\bigcup\nolimits_{{\tiny{\matr {#1} {#2} {#3} {#4}}}\phantom{\Big|}\!\!}}

\begin{document}

\title{Minimal $4$-colored graphs representing an infinite family of hyperbolic $3$-manifolds}

\author{P. CRISTOFORI, E. FOMINYKH, M. MULAZZANI, V. TARKAEV}

\address{Paola CRISTOFORI -  Dipartimento di Scienze Fisiche, Informatiche e Matematiche, Universit\`a di Modena e Reggio Emilia, Italy}
\email{paola.cristofori@unimore.it}

\address{Evgeny FOMINYKH, Vladimir TARKAEV - Krasovskii Institute of Mathematics and Mechanics, Ural Branch of the Russian Academy of Sciences, Ekaterinburg, Russia and Department of Mathematics, Chelyabinsk State University, Chelyabinsk, Russia.}
\email{efominykh@gmail.com\quad v.tarkaev@gmail.com}

\address{Michele MULAZZANI - Dipartimento di Matematica and ARCES,  Universit\`a di Bologna, Italy}
\email{mulazza@dm.unibo.it}

\maketitle

\centerline{\textit{To Professor Maria Teresa Lozano on the occasion of her 70th birthday}}

\begin{abstract} The graph complexity of a compact $3$-manifold is defined as the minimum order among all $4$-colored graphs representing it. 
Exact calculations of graph complexity have been already performed, through tabulations, for closed orientable manifolds (up to graph complexity $32$) and for compact orientable 3-manifolds with toric boundary (up to graph complexity $12$) and for infinite families of lens spaces. 

In this paper we extend to graph complexity $14$ the computations for orientable manifolds with toric boundary and we give  two-sided bounds for the graph complexity of tetrahedral manifolds. As a consequence, we compute the exact value of this invariant for an infinite family of such manifolds.

\bigskip
\medskip


\smallskip

\end{abstract}

\section{Introduction}
Representation tecniques have long been used as an important tool in the study of PL manifolds. The theory of {\it crystallizations}, or more generally of {\it gems}, was introduced as a combinatorial representation of closed PL manifolds of arbitrary dimension by means of a particular class of edge-colored graphs (see \cite{[FGG]}). This tool has been proved to be particularly effective in dimension three adding to classical representation methods such as Heegaard diagrams, spines, framed knots and links, branched coverings, etc...

More recently, the representation by edge-colored graphs has been extended in \cite{[CM]} to non-closed compact $3$-manifolds. 
More precisely, it has been proved that there is a well-defined surjective map from the whole set of {\it $4$-colored graphs} -- i.e., $4$-regular graphs equipped with an {\it edge-coloration} 
(see Subsection \ref{colgraph}) -- to the set of $3$-manifolds that are either closed or have non-empty boundary with no spherical components.

In this context, it is natural to pose the problem of determining and listing 
minimal (with respect to the number of vertices) $4$-colored graphs representing $3$-manifolds. 
The order of a minimal graph $\Gamma$ is called the {\it graph complexity} of the represented manifold $M_\Gamma$. 

By the duality between $4$-colored graphs and a particular kind of vertex-labeled pseudotriangulations (called {\it colored triangulations}), graph complexity of manifolds turns out to be also the number of tetrahedra in a minimal triangulation of this type (see details in Subsection \ref{colgraph}). 

The graph complexity of a manifold is an important invariant in the theory of $3$-manifolds and the problem of its computation is usually very difficult.
Exact values of graph complexity can be obviously computed by enumerating $4$-colored graphs with increasing number of vertices and identifying the represented manifolds. This has been done first in the closed case and more recently 
in the case of non-empty boundary. In particular, there exist tables of 
\begin{itemize}
 \item [(a)] closed orientable $3$-manifolds up to graph complexity $32$ (\cite{[CC],[CC1],[Li]});
 \item [(b)] closed non-orientable $3$-manifolds  up to graph complexity $30$  (\cite{[BCG]}, \cite{[C1]});
 \item [(c)] compact orientable $3$-manifolds with toric boundary up to graph complexity $12$  (\cite{[CFMT]}).
\end{itemize}

As regards the computation of graph complexity for infinite families of $3$-manifolds, few results have been obtained up to now. It is proved in \cite{[CC1]} that lens spaces of the form 
$L(qr + 1,q)$, with $q,r\ge 1$ odd, have graph complexity $4(q + r)$,  
while concrete examples of minimal graphs for the same family are constructed in \cite{[BD]}. 

In Section \ref{sec 4} of this paper we extend table (c) to graph complexity $14$. Moreover, in Section \ref{sect: exact values} we give two-sided bounds for the graph complexity of compact tetrahedral manifolds (i.e., manifolds admitting a triangulation by regular ideal hyperbolic tetrahedra). 
On the basis of this result we construct an infinite family of minimal $4$-colored graphs representing tetrahedral manifolds and, hence, compute the exact value of graph complexity for these manifolds.

\section{Preliminaries}
\subsection{Triangulations}\label{triang} 
Let $\mathcal D = \{\tilde\Delta_1, \ldots, \tilde\Delta_n\}$ be a collection of pairwise-disjoint tetrahedra and suppose $\Phi = \{\varphi_1, \ldots, \varphi_{2n}\}$ is a family of affine homeomorphisms pairing faces of the tetrahedra in $\mathcal D$ so that every face has a unique counterpart. It is allowed that faces in each pair belong either to different tetrahedra or to the same tetrahedron.
We use $\mathcal D/ \Phi$ to denote the space obtained from the disjoint union of the tetrahedra of $\mathcal D$ by identifying all the faces via the homeomorphisms of~$\Phi$. 

It is well known that, by the previous assumptions, the identification space $\mathcal D/ \Phi$ is a $3$-manifold except possibly at the images of some vertices and at the center of some edges of the tetrahedra $\tilde\Delta_i$ under the projection $p: \cup_i \tilde\Delta_i \to \mathcal D/ \Phi.$ 

In the following we restrict our attention to the cases where the singularities of $\mathcal D/ \Phi$ only appear at the images of the vertices.
This happens, for example, when all homeomorphisms of $\Phi$ are orientation-reversing with respect to a fixed orientation of the tetrahedra of $\mathcal D$, and therefore the complement of the singularities is an orientable $3$-manifold.

We collect all these information into a single symbol $\mathcal T$ and call $\mathcal T$ a triangulation of $\mathcal D/ \Phi$; moreover, we also use $|\mathcal T|$ to denote the space $\mathcal D/ \Phi$.   
In the literature this kind of triangulation is often called pseudo- or singular triangulation. 
A tetrahedron, face, edge, or vertex of this triangulation is, respectively, the image of a tetrahedron, face, edge, or vertex of the tetrahedra of $\mathcal D $. 
We will denote the image of the vertices by $\mathcal T^{(0)}$.
 
The link of each vertex of $\mathcal T$ is either a $2$-sphere (such a vertex is called {\it regular}) or a closed surface distinct from the $2$-sphere (such a vertex is called {\it singular}). 
Denote by $\mathcal T^{(0)}_s \subseteq \mathcal T^{(0)}$ the set of the singular vertices of $\mathcal T$. If $\mathcal T^{(0)}_s = \emptyset$, then $\mathcal T$ is a triangulation of the closed orientable $3$-manifold $M = |\mathcal T|$. 
If $\mathcal T^{(0)}_s \neq \emptyset$, we say $\mathcal T\setminus \mathcal T^{(0)}_s$ is a triangulation of the noncompact $3$-manifold $\hat M = |\mathcal T| \setminus \mathcal |\mathcal T^{(0)}_s|$. 
In some cases when $\mathcal T^{(0)}_s = \mathcal T^{(0)}$, then $\mathcal T\setminus \mathcal T^{(0)}_s$ is an ideal triangulation of $\hat M$ (an example are the tetrahedral manifolds in Subsection~\ref{tetrahedral}).
 
Assume that $\mathcal T^{(0)}_s \neq \emptyset$. Let us replace every tetrahedron of $\mathcal T$ by the corresponding partially truncated one, by removing open regular neighborhoods of all singular vertices of $\mathcal T$. In this way we get a compact $3$-manifold $M$ with nonempty boundary. 
It is obvious that we can identify $\text{Int } M = M\setminus\partial M$ with the noncompact $3$-manifold $\hat M = |\mathcal T| \setminus \mathcal |\mathcal T^{(0)}_s|$. 
In this situation, we also say that $\mathcal T\setminus \mathcal T^{(0)}_s$ is a triangulation of the compact $3$-manifold $M$ with nonempty boundary.

\subsection{From $4$-colored graphs to triangulated compact $3$-manifolds}\label{colgraph}

 \begin{definition} \label{$4$-colored graph}
{\em A {\it $4$-colored graph} is a regular $4$-valent multigraph (i.e., multiple edges are allowed, but loops are forbidden) $\Gamma=(V(\Gamma), E(\Gamma))$ endowed with a map $\gamma: E(\Gamma) \rightarrow \C=\{0,1,2,3\}$ that is injective on adjacent edges.
\footnote{Note that there exist (non-bipartite) $4$-regular multigraphs admitting no coloration of this type.}}
\end{definition}

A $3$-dimensional compact manifold $M_\Gamma$, possibly with non-empty non-spherical boundary, can be associated to any $4$-colored graph $\Gamma$ in the following way:

\begin{itemize}
\item consider a collection $\mathcal D(\Gamma)=\{\tilde\Delta_1, \ldots, \tilde\Delta_n\}$ of tetrahedra in bijective correspondence with $V(\Gamma)$ and label the vertices of each tetrahedron by different elements of $\C$;
\item  for each pair of $c$-adjacent vertices of $\Gamma$ ($c\in\C$), glue the faces of the corresponding tetrahedra that are opposite to the $c$-labeled vertices, so that equally labeled vertices are identified;
\item remove from the resulting $3$-pseudocomplex $K(\G)$ small open neighborhoods of the singular vertices.
\end{itemize}

\smallskip

As a consequence of the construction the pseudocomplex $K(\G)$ inherits a natural vertex-labeling by $\C$ that is injective on each simplex. 

We remark that the above construction is dual to the one introduced in \cite{[CM]}, where it is proved that any compact $3$-manifold without spherical boundary components admits a representation by $4$-colored graphs and that the manifold is orientable if and only if the representing graph is bipartite.

\begin{remark} {\em Note that any $4$-colored graph encodes a triangulation in the sense of Subsection \ref{triang}. 
In fact, given the collection of tetrahedra of $\mathcal D(\G)$, the affine homeomorphisms of the triangulation are defined naturally by the gluings of their faces induced by the vertex-labeling.
Therefore, the construction of the pseudocomplex $K(\G)$ is a particular case of the one described in Subsection \ref{triang}.
Note also that in this case no singularities can arise at the images of the centres of the edges.

When the graph is bipartite the tetrahedra of $\mathcal D(\G)$ can be subdivided into two classes according to the bipartition classes of the corresponding vertices of $\G$ and, by giving to the tetrahedra of one class the orientation induced by the cyclic permutation $(0\ 1\ 2\ 3)$ of the labels of their vertices, and to the tetrahedra of the other class the opposite orientation, all the affine homeomorphisms of the triangulation turn out to be orientation-reversing; as a consequence the resulting manifold is orientable.}
\end{remark}

\subsection{Graph and tetrahedral complexities of $3$-manifolds}

 A $4$-colored graph $\Gamma$ is called {\it minimal} if there exists no graph representing $M_{\Gamma}$ with less vertices than $\Gamma$. 

\begin{definition} {\em The {\it graph complexity} of a compact $3$-manifold $M$, denoted by $c_{g}(M)$, is the number of vertices in a minimal $4$-colored graph representing $M$.}\end{definition} 

In case $M$ is a closed manifold a notion of complexity in terms of colored graphs has been already introduced in \cite{[Li]}: it is called gem-complexity, denoted by $k(M)$, and the relation between the two invariants is $c_g(M) = 2k(M)+ 2.$
 
A triangulation of a compact $3$-manifold $M$ into tetrahedra is {\it minimal} if there is no triangulation of $M$ into fewer tetrahedra. The {\it tetrahedral complexity} $c_{tet}(M)$ of $M$ is the number of tetrahedra in a minimal triangulation.

The next result gives an inequality relating the complexities $c_{tet}$ and $c_{g}$.
 
\begin{lemma}
 \label{2complexities}
 For every compact $3$-manifold $M$ we have $c_{tet}(M) \leq c_{g}(M)$.
\end{lemma}

\begin{proof}
 Consider a minimal $4$-colored graph $\G$ representing the manifold $M$. 
 By definition, $\G$ has $c_{g}(M)$ vertices. Therefore, the graph $\G$ determines a triangulation of $M$ with $c_{g}(M)$ tetrahedra. 
 This implies that $c_{tet}(M) \leq c_{g}(M)$. 
\end{proof}

 In Section \ref{sect: exact values} we will apply Lemma \ref{2complexities} in order to find lower bounds for the graph complexity of the
so-called tetrahedral manifolds.

\subsection{Tetrahedral manifolds}\label{tetrahedral}

Let $M$ be a compact $3$-manifold with boundary consisting of tori. Suppose that the interior of $M$, denoted by $Q$, possesses a complete Riemannian metric with finite volume and constant sectional curvature $-1$. 
Following \cite{[FGGTV]}, we say that $M$ is {\it tetrahedral} if there exists a decomposition of $Q$ into ideal regular hyperbolic tetrahedra. Equivalently, there exists an ideal triangulation of $M$ such that each edge class contains exactly six edges of the tetrahedra of~$\mathcal D$.
 
 As mentioned in \cite{[A],[FGGTV],[VTF]}, coverings of tetrahedral manifolds yield infinite families of finite volume hyperbolic $3$-manifolds whose tetrahedral complexity can be calculated exactly. More precisely the following statement holds.
 
\begin{lemma}
 \label{complexityoftetrahedral}
 Let $M$ be a compact tetrahedral manifold such that the interior of $M$ is obtained by gluing together $k$ regular ideal tetrahedra, and let $N$ be an $n$-fold covering of $M$. Then 
 $$ c_{tet}(M) = k \text{\ \ and\ \ } c_{tet}(N) = nk.$$
\end{lemma}

\begin{proof}
Let us denote by $Q$ the interior of $M$. Recall that the volume of the regular ideal tetrahedron, that is $v_{tet} = 1.01494\dots$, is maximal among the volumes of all tetrahedra in $\mathbb H^3$. On this property the relation $c_{tet}(M) \geq \operatorname{vol}(Q)/v_{tet}$ mentioned in \cite{[A]} is based. Since $Q$ is obtained by gluing $k$ regular ideal tetrahedra together, its volume $\operatorname{vol}(Q)$ is $kv_{tet}$. Hence, $c_{tet}(M) = k$.
 
Since the class of tetrahedral manifolds is closed under finite coverings, $N$ is a tetrahedral manifold such that the interior of $N$ is obtained by gluing together $nk$ regular ideal tetrahedra. Hence, $c_{tet}(N) = nk$. 
\end{proof}

\section{Exact values and two-sided bounds for the graph complexity of tetrahedral manifolds}
 \label{sect: exact values}

An  {\it $n$-fold covering} between two $4$-colored graphs $G$ and $\G$, where $n=\#V(G)/\#V(\G)$, is a map $f : V(G)\to V(\G)$ that preserves $c$-adjacency of vertices for all $c\in\C$ (i.e., for each pair of $c$-adjacent vertices $a,b\in V(G)$ the vertices $f(a), f(b)$ are $c$-adjacent in $\G$). 

We call a covering {\it admissible} if it is bijective when restricted to the bicolored cycles of the graphs. 
 
The $n$-fold covering $f$ naturally induces a topological $n$-fold (possibly branched) covering $\vert f\vert\ :M_G\to M_{\G}$.  Moreover, $\vert f\vert$ is unbranched if and only if $f$ is admissible. Note also that the triangulation associated to $G$ is the lifting of the one associated to $\G$.

The next result gives two-sided bounds for the graph complexity of compact tetrahedral manifolds.
 
\begin{theorem}
  \label{twosidedbounds}
 Let $\G$ be a $4$-colored graph with $k$ vertices representing a compact tetrahedral manifold $M_\G$ such that the interior of $M_\G$ is obtained by gluing together $d$ regular ideal tetrahedra. Let $G$ be an admissible $n$-fold covering of $\G$. Then
 $$nd \leq c_{g}(M_{G}) \leq nk.$$
\end{theorem}

\begin{proof}
 Since $G$ is an $n$-fold covering of $\G$, $G$ has $nk$ vertices. This implies that $c_{g}(M_{G}) \leq nk$.

 On the other hand, it follows from Lemma \ref{2complexities} that $c_{tet}(M_{G}) \leq c_{g}(M_{G})$. Since $G$ is an admissible $n$-fold covering of $\G$, $M_{G}$ is an $n$-fold covering of $M_\G$. Thus, by Lemma \ref{complexityoftetrahedral}, we have $c_{tet}(M_{G}) = nd$. 
\end{proof}

 Now we give examples of $4$-colored graphs satisfying the assumptions of Theorem \ref{twosidedbounds}. 
 They allow us to find either the exact values or two-sided bounds for the graph complexity of infinite families of compact tetrahedral manifolds.

\begin{theorem}
 Let $\G$ be the bipartite $4$-colored graph with $12$ vertices of Figure \ref{Gamma_1}. If $G$ is an admissible $n$-fold covering of $\G$, then
 $$c_{g}(M_{G}) = 12n.$$
\end{theorem}

\begin{proof}
 It follows from \cite[Table 3]{[CFMT]} that $\G$ represents the tetrahedral manifold \verb'otet12_00009', which is obtained by gluing together $12$ regular ideal tetrahedra (see details in \cite{[CFMT]}). The conclusion $c_{g}(M_{G}) = 12n$ now follows from Theorem \ref{twosidedbounds}. 
\end{proof}

In Figure \ref{Gamma_2} we give a concrete example of such a graph $G$. As pointed out in \cite{[OVC]}, $M_{G}$ is the complement of the link in $S^3$ composed by the weaving knot ${\mathcal W}(3,3n)$ and its braid axis (see Figure \ref{Link_n}). As in \cite{[CKP]}, the weaving knot ${\mathcal W}(p,q)$ is the alternating knot or link with the same projection as the standard $p$-braid $(\sigma_1\cdots\sigma_{p-1})^q$ projection of the torus knot or link $T(p,q)$.

\begin{figure}[htb]
\centering{
\includegraphics[scale=1.0]{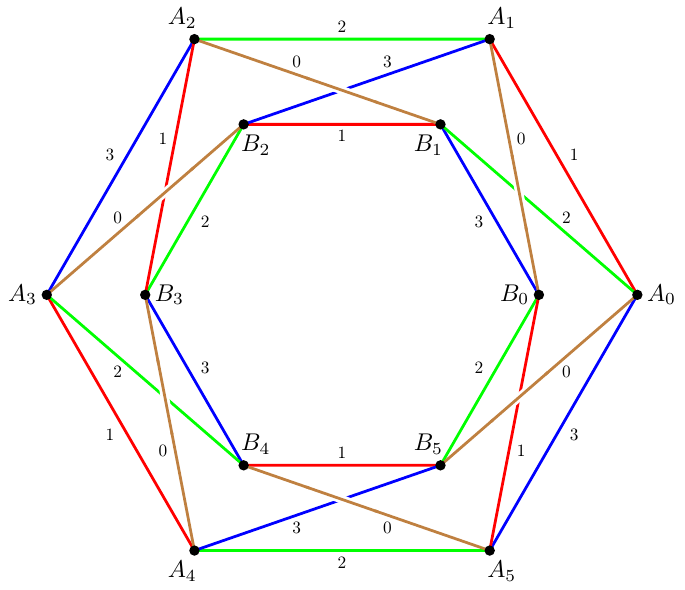}
\caption{The 4-colored graph $\G$.}
\label{Gamma_1}}
\end{figure}

\begin{figure}[htb]
\centering{
\includegraphics[scale=1.0]{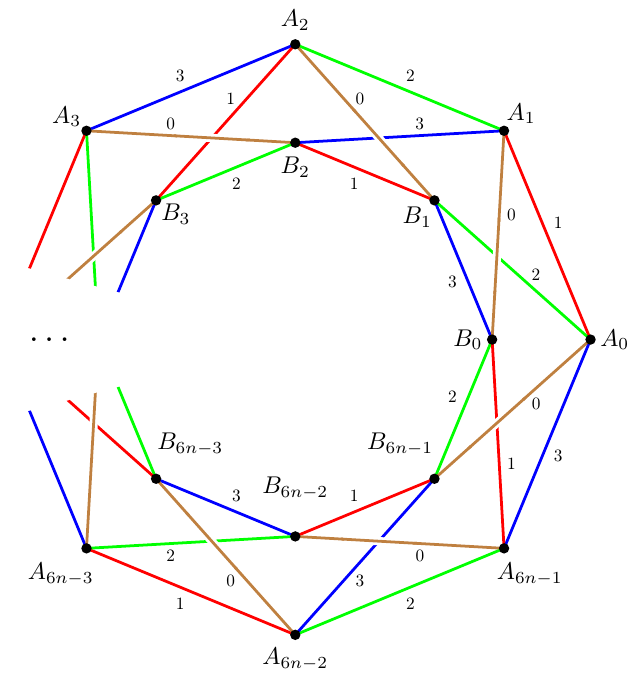}
\caption{An admissible $n$-fold covering of the graph $\G$ depicted in Figure \ref{Gamma_1}.}
\label{Gamma_2}}
\end{figure}

\begin{figure}[htb]
\centering{
\includegraphics[scale=0.8]{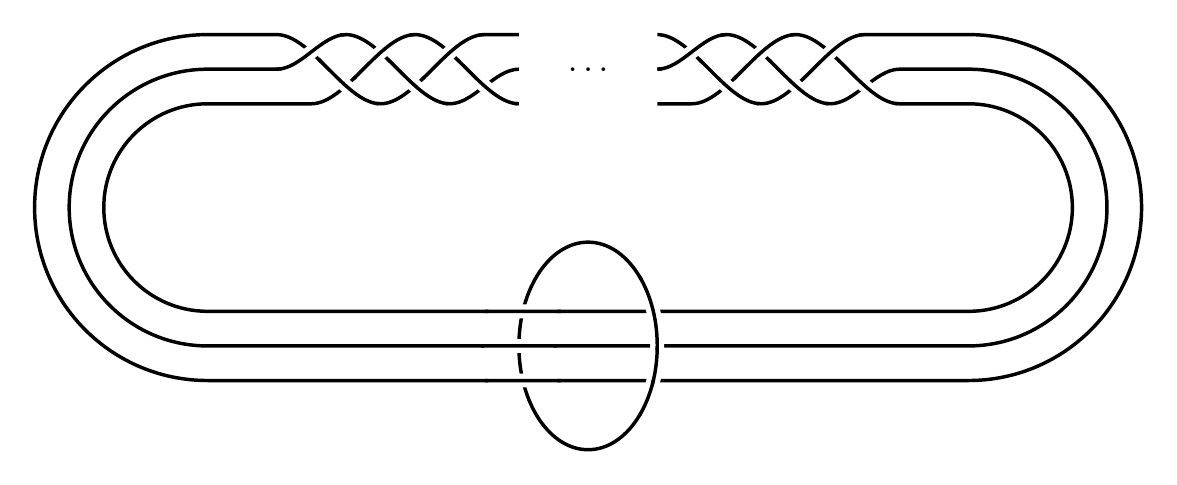}
\caption{The weaving knot ${\mathcal W}(3,3n)$ with the braid axis.}  
\label{Link_n}}
\end{figure}

The {\it code} of a bipartite $4$-colored graph $\G$ with $2p$ vertices is a numerical ``string'' of length $3p$ which completely describes both combinatorial structure and coloration of $\G$.

More precisely, the vertices of $\G$ are divided into the two bipartition classes and labelled by the integers $\{-p,\ldots,-1\}$ and $\{+1,\ldots,+p\}$ respectively. Then, for each $i\in\{1,\dots,p\}$ and $c\in\{1,2,3\}$, the label of the vertex that is $c$-adjacent to $-i$ appears as the $(c-1)p+i$-th character of the string, while $-i$ and $+i$ are assumed to be $0$-adjacent.

Although there are obviously many ways of labeling the vertices and also of permuting the elements of the color set, there exists an algorithm to compute the string such that it uniquely determines $\G$ up to relabeling of the vertices and permutations of the color set (see \cite{[Li]} for details).

When the vertices are few, the code is often displayed by using small letters for negative integers and capital ones for positive integers.

\begin{theorem}

 Let $\G_1, \G_2, \G_3$ be bipartite $4$-colored graphs represented by the following codes: \\
 $\G_1 : DABCFEFEABDCCDEFAB$; \\
 $\G_2 : FABCDEDEFABCCDEFAB$; \\
 $\G_3 : DABCFEFEDABCBCFEDA$. \\
If $G_i$, $1\leq i\leq 3$, is an admissible $n$-fold covering of $\G_i$, then
 $$10n \leq c_{g}(M_{G_i}) \leq 12n.$$
\end{theorem}

\begin{proof}
 It follows from \cite[Table 3]{[CFMT]} that $\G_1, \G_2$ and $\G_3$ represent the tetrahedral manifolds \verb'otet10_00014', \verb'otet10_00028' and \verb'otet10_00027' respectively, which are obtained by gluing together $10$ regular ideal tetrahedra (see details in \cite{[CFMT]}). The double inequality $\ \ 10n \leq c_{g}(M_{G_i}) \leq 12n\ \ $ now follows from Theorem \ref{twosidedbounds}. 
\end{proof}

\section{Manifolds of graph complexity $14$}\label{sec 4}

The previous section shows how it could be useful to have a census of (prime) $3$-manifolds represented by $4$-colored graphs. 

In \cite{[CFMT]}, all prime orientable $3$-manifolds with toric boundary representable by (bipartite) $4$-colored graphs with order $\leq 12$ have been classified.
 
In this section we extend the classification up to $14$ vertices of the associated graphs. 
Moreover, we show that all manifolds appearing in this census, except four, are complements of links in the $3$-sphere whose diagrams are also determined.

The classification has been obtained starting from the catalogues of graphs described in \cite{[CFMT]} by using the programs \verb"3-Manifold" \verb" Recognizer" \cite{[Recognizer]} and \verb"SnapPy" \cite{[SnapPy]} and following the procedure described in the same paper.

\begin{theorem}  
There exist exactly $34$ non-homeomorphic compact orientable prime $3$-manifolds with (possibly disconnected) toric boundary of graph complexity $14$, and exactly $30$ of them are complements of links in the $3$-sphere (see Table~\ref{tab:1}).
\end{theorem}

In order to refer with precision to each manifold in our census, we use a notational system analogous to that used in the knot and link tables. 
For each $1\leq k\leq 5$, we sort in arbitrary order all $3$-manifolds with $k$ boundary components represented by a minimal $4$-colored graph with $14$ vertices, and we denote by $14_n^k$ the $n$-th manifold of this list.

Let us describe which kind of $3$-manifolds can be found in Table~\ref{tab:1}. 
\medskip

\textbf{Seifert manifolds.} A Seifert manifold will be denoted by $(F,(p_1,q_1),\ldots,(p_k,q_k))$, where $F$ is a compact surface with non-empty boundary, $k\geq 0$ and the coprime pairs of integers $(p_i,q_i)$, with $p_i\geq 2$, are the Seifert invariants of the exceptional fibers.

We point out that, by construction, any Seifert manifold with non-empty boundary is endowed with a coordinate system for each of its boundary tori, made by a pair of meridian/longitude suitably oriented. 

All Seifert manifolds appearing in our census, either as single manifolds or as components of a graph manifold, have either disks or M\"obius strips, possibly with holes, as base spaces and at most two exceptional fibers.

In Table~\ref{tab:1}, we denote by $D^2_i$ and $M^2_i$ the disc and the M\"obius strip with $i>0$ holes respectively.
 
\medskip

\textbf{Graph manifolds.} Graph manifolds of Waldhausen are obtained from Seifert manifolds by gluing them along boundary components. 
The structure of the $14$ graph manifolds arising in our census is very simple: each of them is obtained by gluing together either two or three Seifert manifolds as follows. 

\begin{itemize}
 \item Let $M, M'$ be two Seifert manifolds with non-empty boundaries equipped with fixed coordinate systems. Chosen arbitrary tori $T$ and $T'$ of $\partial M$ and $\partial M'$, respectively, let $f_A: T\to T'$, with $A=(a_{ij})\in GL_2(Z)$, be a homeomorphism that takes any curve of type $(m, n)$ on $T$ to a curve of type $(a_{11}m+a_{12}n, a_{21}m+a_{22}n)$ on $T'$. So we define $M\cup_A M'= M\cup_{f_A} M'$. 
\smallskip

 \item Let $M, M', M''$ be three Seifert manifolds with non-empty boundaries equipped with fixed coordinate systems. Chosen arbitrary tori: $T$ of $\partial M$, $T'_1$ and $T'_2$ of $\partial M'$ and $T''$ of $\partial M''$, let $f_A: T\to T'_1$, $f_B: T''\to T'_2$  be homeomorphisms corresponding to the matrices $A, B\in GL_2(Z)$ as above, then we define $M\cup_A M'\cup_B M''=M\cup_{f_A} M'\cup_{f_B} M''$. 

\end{itemize} 

\textbf{Hyperbolic manifolds.} Of the seven hyperbolic manifolds in our census, three ($14^3_9,\ 14^3_{10}$ and $14^4_{14}$), by removing their boundary, give rise to cusped hyperbolic $3$-manifolds that are contained in the orientable cusped census \cite{[CHW]} or in the censuses of Platonic manifolds of \verb"SnapPy". 

Therefore they are identified, in Table~\ref{tab:1}, by the notations of their corresponding cusped manifolds. 

\medskip

\textbf{Composite manifolds.} We call a $3$-manifold {\it composite} if its JSJ decomposition is non-trivial and contains a hyperbolic manifold. 
Each of the $10$ composite manifolds arising in our census is obtained by gluing together one hyperbolic manifold and either one or two Seifert manifolds as follows. 

\begin{itemize}
 \item Let $M$ be a Seifert manifold with non-empty boundary equipped with fixed coordinate systems as remarked above. 
 Let $M_L$ be a hyperbolic manifold, which is the complement of an open regular neighbourhood of a link $L = L_1\sqcup\ldots\sqcup L_r$ in $S^3$. 
 A preferred coordinate system for $\partial M_L$ can be also chosen in the following way. 
 On the regular neighbourhood of each $L_i$, considered as a knot in $S^3$, we choose a standard coordinate system formed, as usual, by the boundary of a meridian disk and a homologically trivial curve in the complement of $L_i$. 
 Therefore, once a boundary torus $T$ of $M$ and an $i$-th component $(\partial M_L)_i$ of $\partial M_L$ corresponding to $L_i$ are chosen, a homeomorphism $f_{A, i}\ :\ T\to (\partial M_L)_i$ can be described by means of a matrix $A\in GL_2(Z)$ as in the case of graph manifolds. 
 Finally, we denote by $M\cup_{A, i} M_L$ the manifold obtained by gluing $M$ and $M_L$ through the homeomorphism $f_{A, i}$. Since in Table~\ref{tab:1} each manifold $M_L$ is represented by a link with up to $8$ crossings, we numerate the components $L_1, \ldots , L_r$ of $L$ as they appear in its Gauss code displayed in the corresponding page of \cite{[KnotAtlas]}.
 \item {Given two Seifert manifolds with non-empty boundaries $M'$ and $M''$ and a hyperbolic manifold $M_L$ as above, we denote by $M\cup_{A, i} M_L\cup_{B, j} M''$ the manifold obtained by identifying two bondary tori of $M'$ and $M''$ with $(\partial M_L)_i$ and $(\partial M_L)_j$ respectively by the homeomorphisms $f_A$ and $f_B$ similarly to the previous case.}
\end{itemize}

All prime links appearing in Table \ref{tab:1} are contained in the Thistlethwaite link table up to $14$ crossings distributed with \verb"SnapPy"; they are identified through their Thistlethwaite name, that is of the form $L[k]a[j_1]$ or $L[k]n[j_2]$, depending on whether the link is alternating or not. Here $k$ is the crossing number and $j_1, j_2$ are archive numbers assigned to each $(a, k)$, $(n, k)$ pair, respectively.
All other links of Table  \ref{tab:1} are not prime and their diagrams are depicted in Figure \ref{links}.

\begin{longtable}{|>{\tiny}c||>{\tiny}c|>{\tiny}c|>{\tiny}c|>{\tiny}c|}
\caption{Orientable prime $3$-manifolds with toric boundary of graph complexity $14$} \label{tab:1} \\ 
  \hline
   Name & Code & Manifold & Link  \\
 \hline \hline \endhead
 \hline 
 $14^2_1$ & EABCDGFGDFEBCADGEFBAC  & $\seifuno {D^2_1}31$ & L6a3 \\  
 \hline
 $14^2_2$ & DABCGEFGFECDBABGDFACE & $\seifdue {D^2}2131 \bigu 1110 (D^2_2\times S^1)$ & see fig. \ref{links} \\ 
 \hline
 $14^2_3$ & GABCDEFEDGFABCDEFAGCB &  $\seifuno {D^2_1}21 \bigu 0110 \seifuno {D^2_1}21$ & L11n204 \\  
 \hline
 $14^3_1$ & EABCDGFGEFCADBCEGAFBD & $\seifuno {D^2_2}21$ & L12n1998 \\ 
 \hline
 $14^3_2$ & DABCGEFGFBADCEFCEAGDB & $\seifuno {M^2_2}10$ & -- \\  
 \hline
 $14^3_3$ & DABCGEFFDBECGAEDGCFAB & $\seifdue {D^2}2131 \bigu 1110 (D^2_3\times S^1)$ & see fig. \ref{links} \\  
 \hline
 $14^3_4$ & EABCDGFGDFEBCABDGAFEC & $\seifuno {D^2_1}21 \bigu 1211 (D^2_2\times S^1)$ & L8n6 \\ 
 \hline
  $14^3_5$ & DABCGEFGEFBDACFGEBACD & $\seifuno {D^2_1}31 \bigu 0110 (D^2_2\times S^1)$ & see fig. \ref{links} \\  
 \hline
 $14^3_6$ & EABCDGFGFDABECCEFAGDB & $\seifuno {M^2_1}10 \bigu 0110 (D^2_2\times S^1)$ & -- \\ 
 \hline
 $14^3_7$ & EABCDGFGFDABECFDGBACE & $\seifuno {D^2_1}21 \bigu 0110 (D^2_2\times S^1) \bigu 0110 \seifuno {D^2_1}21$ & see fig. \ref{links} \\ 
 \hline
 $14^3_8$ & DABCGEFGDFCABEBFDECGA & $(D^2_2\times S^1) \bigcup\nolimits_{{\tiny{\matr {0} {1} {1} {0}}}, 1} M_{L5a1}$ & L13n9356 \\ 
 \hline
 $14^3_9$ & DABCGEFGEFBDACFCGABDE & \verb't12066', \verb'ooct02_00003' & L8n5 \\ 
 \hline
 $14^3_{10}$ & DABCGEFGDFBACECFAEDGB & \verb't12067', \verb'ooct02_00005' & L6a4 \\ 
 \hline
 $14^4_1$ & EABCDGFGFBACEDBCFDGAE & $(D^2_2\times S^1) \bigu 0110 \seifuno {D^2_2}21$ & see fig. \ref{links} \\  
 \hline
 $14^4_2$ & EABCDGFGBEDFACEFAGCDB & $(D^2_2\times S^1) \bigu 0110 \seifuno {M^2_2}10$ & -- \\ 
 \hline
 $14^4_3$ & DABCGEFGCFADBEEGABCFD &  $\seifuno {D^2_1}21 \bigu 0110 (D^2_3\times S^1) \bigu 0110 \seifuno {D^2_1}21$ & see fig. \ref{links} \\  
 \hline
 $14^4_4$ & EABCDGFGEFCADBBFDGEAC & $(D^2_2\times S^1) \bigu 0110 \seifuno {D^2_1}21 \bigu {-1}21{-1} (D^2_2\times S^1)$ & L11n379 \\ 
 \hline
 $14^4_5$ & EABCDGFGFECABDCGDAFEB & $\seifuno {D^2_1}21 \bigu 1211 (D^2_2\times S^1) \bigu 0110 (D^2_2\times S^1)$ & see fig. \ref{links} \\  
 \hline
 $14^4_6$ & EABCDGFGDFACEBBFEDGAC & $\seifuno {M^2_1}10 \bigu 0110 (D^2_2\times S^1) \bigu 0110 (D^2_2\times S^1)$ & -- \\  
 \hline
 $14^4_7$ & EABCDGFGFDEBCAFDEGCAB & $\seifuno {D^2_1}21 \bigcup\nolimits_{{\tiny{\matr {1} {0} {0} {-1}}}, 1} M_{L8n7}$ & L14n63157 \\ 
 \hline
  $14^4_8$ & EABCDGFGFEBCDACGFEBAD & $\seifuno {D^2_1}21 \bigcup\nolimits_{{\tiny{\matr {0} {1} {1} {0}}}, 1} M_{L8n7}$ & L14n61549 \\  
 \hline
 $14^4_9$ & EABCDGFGDFEBCAFCGADEB & $(D^2_2\times S^1) \bigcup\nolimits_{{\tiny{\matr {0} {1} {1} {1}}}, 1} M_{L6a5}$ & L14n62850 \\ 
 \hline
 $14^4_{10}$ & EABCDGFGFBEACDDCGAFEB & $(D^2_2\times S^1) \bigcup\nolimits_{{\tiny{\matr {0} {1} {1} {0}}}, 3} M_{L8n5}$ & see fig. \ref{links} \\ 
 \hline
 $14^4_{11}$ & DABCGEFGEFBDACFGCABDE & $(D^2_2\times S^1) \bigcup\nolimits_{{\tiny{\matr {0} {1} {1} {0}}}, 1} M_{L8n5}$ & L14n62541 \\ 
 \hline
 $14^4_{12}$ & EABCDGFGEFBDACCGAEFDB & $(D^2_2\times S^1) \bigcup\nolimits_{{\tiny{\matr {0} {1} {1} {0}}}, 1}  M_{L6a4}$ & see fig. \ref{links} \\ 
 \hline
 $14^4_{13}$ & EABCDGFGEFBDACBGCEFDA & $(D^2_2\times S^1) \bigcup\nolimits_{{\tiny{\matr {0} {1} {1} {0}}}, 1} M_{L5a1} \bigcup\nolimits_{{\tiny{\matr {0} {1} {1} {0}}}, 2} (D^2_2\times S^1)$ & see fig. \ref{links} \\ 
 \hline
 $14^4_{14}$ & EABCDGFGDFACEBFCEGBAD & \verb'otet10_00011', \verb'ocube02_00044' & L8a21 \\ 
 \hline
 $14^4_{15}$ & EABCDGFGFDABECFDEGCAB & hyperbolic manifold with $\operatorname{Vol} = 10.6669791338$ & L14n60227  \\  
 \hline
 $14^4_{16}$ & EABCDGFGFEACBDCDFGAEB & hyperbolic manifold with $\operatorname{Vol} = 11.202941612$ & L10n96 \\   
 \hline
 $14^4_{17}$ & DABCGEFGEFBDACCGAFBDE & hyperbolic manifold with $\operatorname{Vol} = 12.8448530047$ & L11n456 \\  
 \hline
 $14^4_{18}$ & DABCGEFGEFBDACFGEACDB & hyperbolic manifold with $\operatorname{Vol} = 12.3173273072$ & L14n63000 \\  
 \hline
 $14^5_1$ & EABCDFGGFEBADCCDEGFAB & $(D^2_2\times S^1) \bigu 0110 (D^2_3\times S^1)$ & see fig. \ref{links} \\ 
 \hline 
 $14^5_2$ & DABCGEFGFBADCEECFGABD & $(D^2_2\times S^1) \bigcup\nolimits_{{\tiny{\matr {1} {0} {0} {-1}}}, 1} M_{L8n7}$ & L12n2249 \\ 
\hline
 $14^5_3$ & DABCGEFGCFADBECDEGAFB & $(D^2_2\times S^1) \bigcup\nolimits_{{\tiny{\matr {0} {1} {1} {0}}}, 1} M_{L8n7}$ & L14n63769 \\ 
 \hline 

\end{longtable}


\begin{figure*}[h]
\centering{
\includegraphics[height=0.11\textheight]{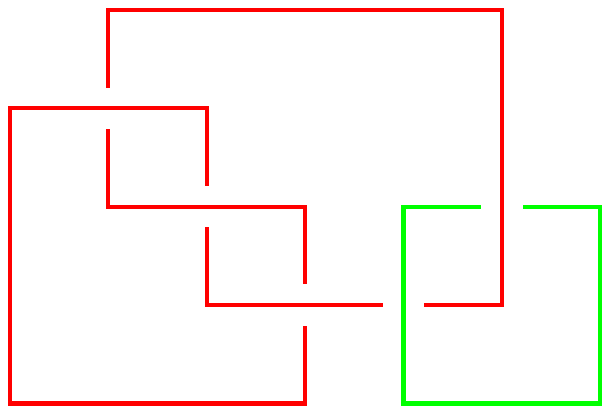}\hspace{10mm} 
\includegraphics[height=0.11\textheight]{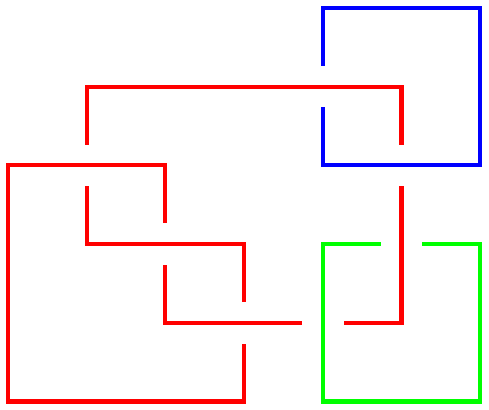}\hspace{10mm}
\includegraphics[height=0.11\textheight]{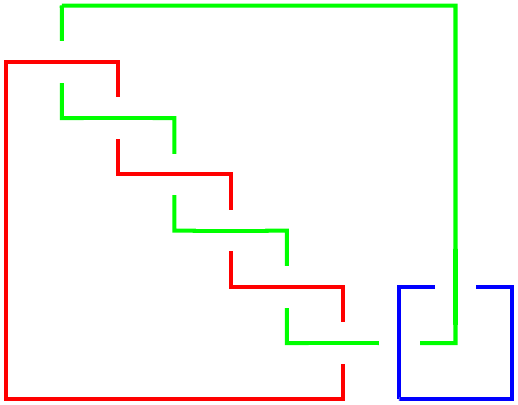}\hspace{10mm}
$$
14^2_2 \qquad \qquad \qquad \qquad \quad
14^3_3 \qquad \qquad \qquad \qquad \quad
14^3_5 
$$
\medskip
\\
\includegraphics[height=0.11\textheight]{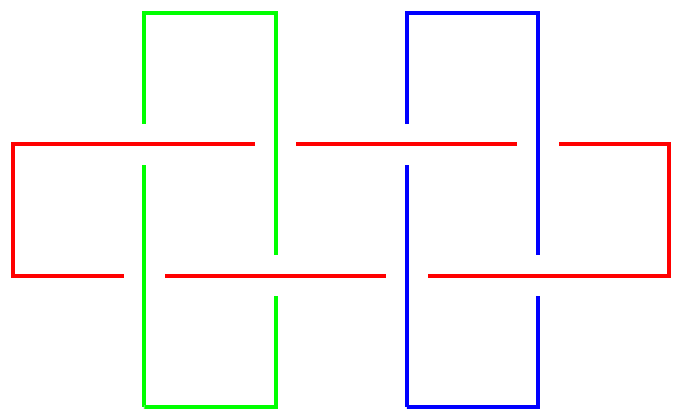}\hspace{10mm}
\includegraphics[height=0.11\textheight]{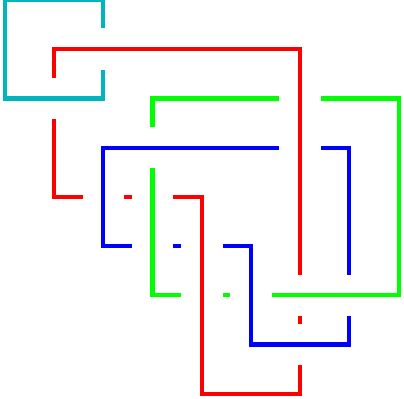}\hspace{10mm}
\includegraphics[height=0.11\textheight]{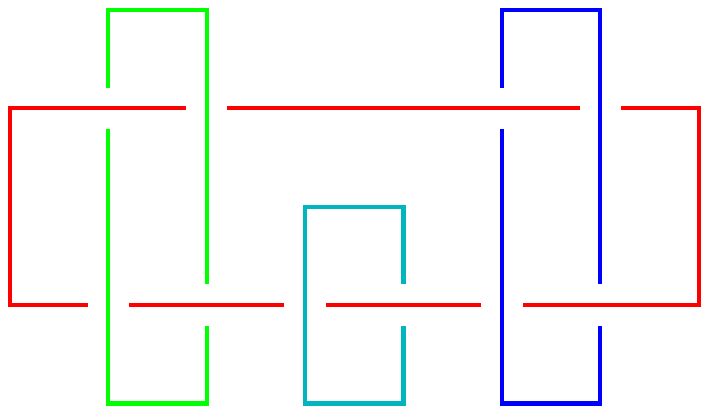}\hspace{10mm}
$$
14^3_7 \qquad \qquad \qquad \qquad \quad
14^4_1 \qquad \qquad \qquad \qquad \quad
14^4_3
$$
\medskip
\\
\includegraphics[height=0.11\textheight]{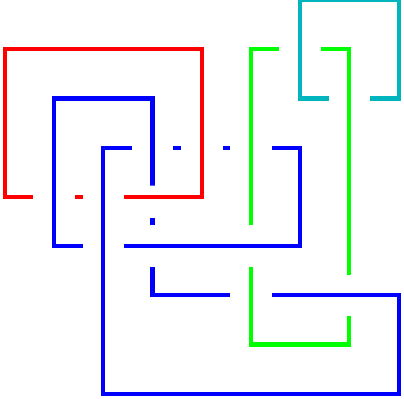}\hspace{10mm} 
\includegraphics[height=0.11\textheight]{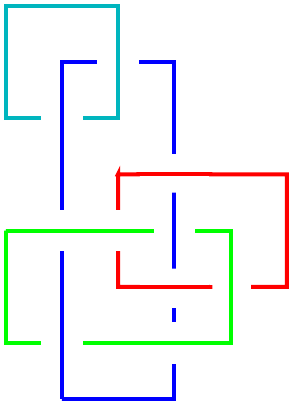}\hspace{10mm}
\includegraphics[height=0.11\textheight]{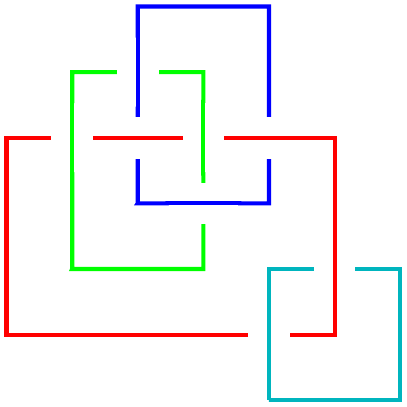}\hspace{10mm}
$$
14^4_5  \qquad \qquad \qquad \qquad \quad
14^4_{10}  \qquad \qquad \qquad \qquad \quad
14^4_{12} 
$$
\\
\includegraphics[height=0.12\textheight]{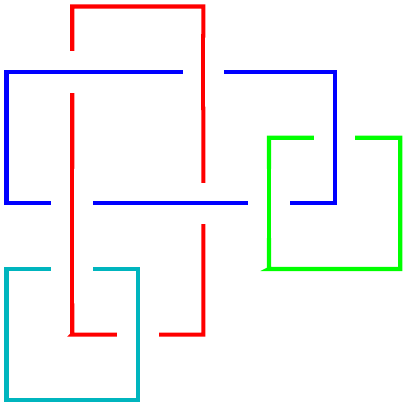}\hspace{10mm} 
\includegraphics[height=0.12\textheight]{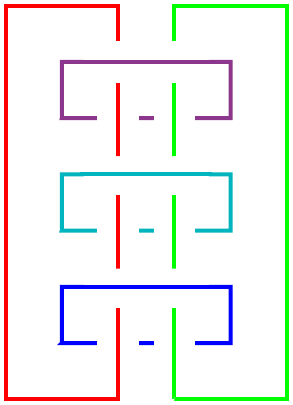}\hspace{10mm}
$$
14^4_{13}  \qquad \qquad \qquad \qquad \quad
14^5_1  
$$
\caption{Non-prime links with complements represented by $4$-colored graphs of order $14.$}
\label{links}
}
\end{figure*}

\textbf{Acknowledgements:}  P. Cristofori and M. Mulazzani have been supported by the National Group for
Algebraic and Geometric Structures, and their Applications'' (GNSAGA-INdAM), the University of Modena and Reggio Emilia and the University of Bologna, funds for selected research topics.
E. Fominykh and V. Tarkaev have been supported by RFBR (grant number 16-01-00609).

\end{document}